\documentclass{article}
\usepackage{amsmath}
\usepackage{amsthm}
\usepackage{amsfonts}
\usepackage{amssymb}
\usepackage{hyperref}
\usepackage{graphicx}
\usepackage{geometry}
\usepackage{indentfirst}

\newtheorem{theorem}{Theorem}

\newtheorem{lemma}{Lemma}

\theoremstyle{remark}

\begin{document}

\title{On Continuity Properties for Infinite Rectangle Packing}
\author{Zhiheng Liu\\
University of Science and Technology of China\\
liuzhih@mail.ustc.edu.cn}

\maketitle

\begin{abstract}
By rectangle packing we mean putting a set of rectangles into an enclosing rectangle, without any overlapping. We begin with perfect rectangle packing problems, then prove two continuity properties for parallel rectangle packing problems, and discuss how they might be used to obtain negative results for perfect rectangle packing problems.
\end{abstract}

\section{Introduction}

By rectangle packing we mean putting a set of rectangles into an enclosing rectangle, without any overlapping, in which most famous is \cite{knuth} the following problem: since $\sum_{n=1}^{\infty}\frac{1}{n\times(n+1)}=1$, whether the set of rectangles of dimension $\frac{1}{n}\times\frac{1}{n+1}$ can be packed into the unit square? There are many variants of this problem, which we will call perfect rectangle packing problems, that is, a packing with no wasted space. These problems have got some researches, among them are \cite{paulhus,chalcraft,wastlund}, which use explicit strategies/algorithms to do the packing. In particular, Chalcraft \cite{chalcraft} has proved: there is a perfect packing of the squares of side $n^{-t}$ into a square, where $1/2<t\leqslant3/5$, provided a certain algorithm succeeds for that value of $t$. 

We can intuitively interpret this result as follows. If we regard the sides of the set of rectangles as a configuration space, (that is, each series of rectangles represent a point in this space,) and the family of series of squares with sides $n^{-t}$ as a curve parametrized by $t$ in this space, then this result says: this curve lies in the "zone of perfect packing", as long as $1/2<t\leqslant3/5$ and a certain algorithm succeeds for that value of $t$. This idea is one motivation of the results to be stated in the next section.

Besides, we notice Altunkaynak \cite{baris} has turned the perfect rectangle packing problem into an equivalent algebraic problem.

For simplicity and fluency we will confine to (fixed-direction) parallel rectangle packing, i.e.\ rectangle packing with the sides of rectangles parallel to the enclosing rectangle and whose directions are prescribed. It seems not hard to generalize to non-parallel cases and other shapes. 

\section{First Continuity Property}

Let $A$ be a set of rectangles $A_1,A_2,...,A_n$ of size $w_1\times{l_1},w_2\times{l_2},...,w_n\times{l_n}$ (as ordered pairs), let $S(A)=\sum_{i=1}^{n}w_il_i$ to be the total area of $A$. (We do \textit{not} suppose $w_i\leqslant{l}_i$ throughout this section.)

By a \emph{positioning} of $A$, we mean putting $A$ onto the coordinate plane, such that the interior of $A_i$ does not overlap, the sides of $A_i$ are parallel to coordinate axes, and in particular, the sides with length $w_i$ parallel to $x$-axis.

If $A_M$ is a positioning of $A$, $A_M$ contains the information of the coordinates of $A_i$. Let's say the coordinate of the lower left corner of $A_i$ is $(x_i^-,y_i^-)$, the upper right corner $(x_i^+,y_i^+)$. Define $p(A_M)=\max_{i,j}(x_i^+-x_j^-)$, $q(A_M)=\max_{i,j}(y_i^+-y_j^-)$, and $T(A_M)=p(A_M){q(A_M)}$ to be the "bounding area" of $A_M$. Define \emph{packing efficiency} to be $\eta(A_M)=S(A)/T(A_M)$.

We define \emph{best packing efficiency} of $A$ to be $\eta_0(A)=\sup{\eta(A_M)}$, where the sup is over all positioning of $A$. The fact that this sup can be attained (i.e.\ there exists a positioning of $A$ which has the best packing efficiency) is because, see \cite{martin}.

For \textit{every} $A$, choose a positioning of $A$ with the best packing efficiency, denote it as $A_0$, and fix our choice. Since $T(A_0)=S(A)/\eta_0(A)$, which is irrelevant of the choice, we define $T_0(A)=T(A_0)$. (Alternatively, $T_0(A)=\inf{T}(A_M)$.)

Two operations on the positioning will be key to the proof of the results. Let $A_M$ be a positioning of $A$, with coordinates $x_i^-,y_i^-,x_i^+,y_i^+$.

(I)\emph{retraction}: Let $\Delta{x}$ be a number such that $0<\Delta{x}<w_1$. We modify $A_M$ as follows: let the coordinates of $A_1$ be $x_1^-,y_1^-,(x_1^+-\Delta{x}),y_1^+$, and other coordinates remain the same. This is a retraction (of x-side of $A_1$ in $A_M$).

(II)\emph{extension}: Let $\Delta{x}$ be a number such that $0<\Delta{x}$. We modify $A_M$ as follows: let the coordinates of $A_1$ be $x_1^-,y_1^-,(x_1^++\Delta{x}),y_1^+$. For $k>1$, if $x_k^-\geqslant{x}_1^+$, let the coordinates of $A_k$ be $(x_k^-+\Delta{x}),y_k^-,(x_k^++\Delta{x}),y_k^+$, otherwise remain the same. This is an extension (of x-side of $A_1$ in $A_M$).

After each operation, since the size of $A_1$ has changed, denote it as $A'_1$. It's not hard to verify that a retraction or extension of $A_M$ as described above is indeed a positioning of $A'_1,A_2,...,A_n$.

We need the following:

\begin{lemma}\label{estimation}
(i)Denote $W$ as the retraction of $A_M$ described above. Then $T(W)\leqslant{T}(A_M)$.\\
(ii)Denote $W$ as the extension of $A_M$ described above. Then $T(W)\leqslant{T}(A_M)+q(A_M)\Delta{x}$.\\
(iii)Denote $W$ as the extension of $A_M$ described above. Further extend y-side of first rectangle of $W$ by $\Delta{y}$. Denote this extension as $V$. Then $T(V)\leqslant{T}(A_M)+q(A_M)\Delta{x}+p(A_M)\Delta{y}+\Delta{x}\Delta{y}$.
\end{lemma}

\begin{proof}
(i) follows from $p(W)\leqslant{p}(A_M)$, $q(W)=q(A_M)$.
(ii)$\&$(iii) follows from $p(W)\leqslant{p}(A_M)+\Delta{x}$, $q(W)=q(A_M)$, $p(V)=p(W)$, $q(V)\leqslant{q}(W)+\Delta{y}$.
\end{proof}

From now on we denote, as vectors, $A=(w_1,l_1,...,w_n,l_n)$, $\Delta{A}=(\Delta{w}_1,\Delta{l}_1,...,\Delta{w_n},\Delta{l_n})$.

\begin{lemma}\label{p_0}
Let $A'=A+\Delta{A}$. Then $p(A'_0)\leqslant\sum_{i=1}^n{w_i}+\sum_{i=1}^n{|\Delta{w_i}|}$. The similar for q.
\end{lemma}

\begin{proof}
Let $B=\{A'_\omega\text{ is a positioning of }A'\mid\eta(A'_\omega)=\eta_0(A')\}$. We argue that there cannot be $A'_\omega\in{B}$ such that $p(A'_\omega)>\sum_{i=1}^n{w_i}+\sum_{i=1}^n{|\Delta{w_i}|}$. (Since $A'_0\in{B}$, this will complete the proof.) Otherwise, suppose such $A'_\omega$ exists. We construct a positioning of $A'$ as follows: Consider the projection of the rectangles in $A'_\omega$ onto the $x$-axis. Since $p(A'_\omega)>\sum_{i=1}^n{w_i}+\sum_{i=1}^n{|\Delta{w_i}|}$, there must be gaps in between. "Squeeze" them out, we get a new positioning of $A'$, denote it as $A'_L$. Since $p(A'_L)<p(A'_\omega)$, $q(A'_L)=q(A'_\omega)$, we have $T(A'_L)<T(A'_\omega)$. But this contradicts the fact $T(A'_\omega)=T_0(A')$.
\end{proof}

\begin{lemma}\label{t_0}
Let $A'=A+\Delta{A}$. Then $T_0(A')\rightarrow{T_0(A)}\ (\Delta{A}\rightarrow\mathbf{0})$.
\end{lemma}

\begin{proof}
For convenience, denote $T_0(A')$ as $T'_0$, $T_0(A)$ as $T_0$, let $\Delta{T_0}=T'_0-T_0$. Let $I_w$,$J_w$,$I_l$,$J_l$ be subset of the set $\{1,2,...,n\}$, such that $I_w$ and $J_w$ is disjoint, $I_l$ and $J_l$ is disjoint. Let $Z(I_w,J_w,I_l,J_l)=\{(z_{11},z_{12},...,z_{n1},z_{n2})\in\mathbb{R}^{2n}\mid\text{if }k\in{I_w},z_{k1}>0,\text{if }k\in{J_w},z_{k1}<0,\text{else }z_{k1}=0,\text{ the similar for }z_{k2}\}$. Let $Z(A)=\{(z_{11},z_{12},...,z_{n1},z_{n2})\mid{w_k+z_{k1}>0, l_k+z_{k2}>0}\text{ for all }k\}$.(which is the natural boundary condition for $\Delta{A}$) There are two cases:

(I)$\Delta{T_0}\geqslant0$. We construct a positioning of $A'$ as follows: make retractions and extensions on $A_0$. If $k\in{I_w}$, make an extension on $w_k$ of $A_0$, if $k\in{J_w}$, make a retraction on $w_k$ of $A_0$, the similar for $l_k$. Denote this positioning of $A'$ as $A'_w$. For $\Delta{A}\in{Z}(I_w,J_w,I_l,J_l)\cap{Z}(A)$ and satisfies the condition for this case, we have
%\begin{equation}
\begin{multline}
\nonumber
|\Delta{T_0}|=\Delta{T_0}\leqslant{T(A'_w)}-T_0\\
\leqslant{T}_0+q(A_0)\sum_{k\in{I_w}}\Delta{w_k}+p(A_0)\sum_{k\in{I_l}}\Delta{l_k}+(\sum_{k\in{I_w}}\Delta{w_k})(\sum_{k\in{I_l}}\Delta{l_k})-T_0\quad\text{(by Lemma \ref{estimation})}\\
=q(A_0)\sum_{k\in{I_w}}\Delta{w_k}+p(A_0)\sum_{k\in{I_l}}\Delta{l_k}+(\sum_{k\in{I_w}}\Delta{w_k})(\sum_{k\in{I_l}}\Delta{l_k})\\
\end{multline}
%\end{equation}
So, for all $\Delta{A}\in{Z}(A)$, as long as $\Delta{T_0}\geqslant0$, $|\Delta{T_0}|\leqslant{q}(A_0)\sum_{k=1}^n|\Delta{w_k}|+p(A_0)\sum_{k=1}^n|\Delta{l_k}|+(\sum_{k=1}^n|\Delta{w_k}|)(\sum_{k=1}^n|\Delta{l_k}|)$.

(II)$\Delta{T_0}<0$. The situation is similar except that we construct a positioning of $A$ based on $A'_0$, and to deal with $p(A'_0)$ and $q(A'_0)$ we need to use Lemma \ref{p_0}. We get, for all $\Delta{A}\in{Z}(A)$, as long as $\Delta{T_0}<0$, $|\Delta{T_0}|\leqslant(\sum_{k=1}^n{l_k}+\sum_{k=1}^n{|\Delta{l_k}|})(\sum_{k=1}^n|\Delta{w_k}|)+(\sum_{k=1}^n{w_k}+\sum_{k=1}^n{|\Delta{w_k}|})(\sum_{k=1}^n|\Delta{l_k}|)+(\sum_{k=1}^n|\Delta{w_k}|)(\sum_{k=1}^n|\Delta{l_k}|)$.

Combine these two cases, we have for all $\Delta{A}\in{Z}(A)$, the estimation as stated in case (II) is true, which would derive the result. %(Apply Lemma \ref{p_0} to the estimation in case (I), we see it's  contained in that of case (II).)
\end{proof}

\begin{theorem}
$\eta_0:\mathbb{R}_{>0}^{2n}\mapsto[0,1]: (w_1,l_1,w_2,l_2,...,w_n,l_n)\mapsto\eta_0$ is a continuous map.
\end{theorem}

\begin{proof}
By definition, $\eta_0(A)=S(A)/T_0(A)$, and S is obviously continuous, $T_0$ is continuous by Lemma \ref{t_0}, thus the result.
\end{proof}

%\begin{remark}
%An infinite version will not be stated since that may involve topology issues.
%\end{remark}

\section{Second Continuity Property}

\textbf{(The meaning of some notations will be changed in this section, as described below:)}

Let $A$ be a set of rectangles $A_1,A_2,...,A_n,...$ of size $w_1\times{l_1},w_2\times{l_2},...,w_n\times{l_n},...$ satisfying following conditions (to rule out irregular cases):

(C1) $\sum_{i=1}^{\infty}l_i^2$ is finite

(C2) $w_i\leqslant{l_i}$, and $l_{i+1}\leqslant{l}_i$

Let $S(A)=\sum_{i=1}^{\infty}w_il_i$. Let $A^n$ denote the first n rectangles of $A$.

If $A_M$ is a positioning of $A$, define $p(A_M)$, $q(A_M)$, $T(A_M)$, $\eta(A_M)$, $\eta_0(A)$, $A_0$, $T_0(A)$ the same way as before. If any of these is infinite or does not exist, leave it as undefined.

\begin{theorem}
$\eta_0(A)$ is well-defined and $\eta_{0}(A^n)\rightarrow\eta_{0}(A)\ (n\rightarrow\infty)$.
\end{theorem}

\begin{proof}
Since $\eta_{0}(A^n)=S(A^n)/T_0(A^n)$, $\eta_{0}(A)=S(A)/T_0(A)$, and $S(A^n)\rightarrow{S}(A)$, we only need to prove $T_0(A)$ is finite and $T_0(A^n)\rightarrow{T_0(A)}$. We construct a positioning of $A$ as follows:

We use Theorem 4 in \cite{meir} by Meir and Moser. Using their method we put all remaining rectangles into a box of size $b\times{a}$ on the right side of $(A^n)_0$. (By our definition of positioning, condition (C2), and the construction in the proof of their theorem, the "a" side in their theorem must be parallel to our $y$-axis.) Denote this positioning as $A_w$. Denote $R_n=\sum_{k=n+1}^{\infty}{l_k^2}$, which $\rightarrow0$ by condition (C1). If $a$ satisfies $l_{n+1}\leqslant{a}\leqslant{q}((A^n)_0)$, we have
\begin{multline}
\nonumber
T(A_w)-T_0(A^n)\leqslant{q((A^n)_0)}[({2\sum_{k=n+1}^{\infty}{w_kl_k}+a^2/8})/a]\quad\text{by assumption on $a$ and Theorem 4 in \cite{meir}}\\
\leqslant[{T_0(A^n)}/w_1][({2R_n+a^2/8})/a]\quad\text{by $p((A^n)_0)\geqslant{w_1}$ and condition (C2)}
\end{multline}
 We need to choose an appropriate $a$. Choose $a=4\sqrt{R_n}$ (which minimizes). $4\sqrt{R_n}\geqslant{l}_{n+1}$ is obvious, and since $q((A^n)_0)\geqslant{l_1}$ and $R_n\rightarrow0$, $4\sqrt{R_n}\leqslant{q}((A^n)_0)$ will be true for all $n$ sufficiently large. Then $T_0(A)-T_0(A^n)\leqslant{T(A_w)-T_0(A^n)}\leqslant{T_0(A^n)}\sqrt{R_n}/w_1$ for all $n$ sufficiently large.

Now $\eta_0(A)$, $A_0$, $T_0(A)$ are well-defined. It remains to show $T_0(A^{n})\leqslant{T}_0(A)$. Construct a positioning of $A^n$ as follows: remove all the k-th rectangles from $A_0$ where $k>n$. Denote this positioning as $A^n_v$. Then $T_0(A^n)\leqslant{T}(A^n_v)\leqslant{T}_0(A)$.
\end{proof}

Remark: alternatively, one can prove this by pretending all remaining rectangles to be squares, and put them into a box utilizing Theorem 1 in \cite{meir} by Meir and Moser, and this approach would easily generalize to higher dimension packing and other shapes.

\section{Discussion}

(Methods of) theorem 2 enables us to compute $\eta_0(A)$ from $\eta_0(A^n)$ with error terms. (Methods of) theorem 1 can extend the result of $\eta_0$ at one point in the configuration space to nearby points, with error terms. (note that $\eta$ is invariant under scaling, so one can first do a scaling, then use the estimation.)

The perfect packing problem is just $\eta_0^{-1}(1)$. It does not seem easy to determine whether there does not exist perfect packing (unless the enclosing rectangle is required to be a specific shape) before, while the ability to compute $\eta_0$ with error terms will mean that can be determined in finite steps (as long as $\eta_0\neq1$), thus can give negative results on perfect packing problems. Since the task of computing $\eta_0$ is hard \cite{huang}, it might be not practical to compute it, but the methods are still applicable to weaker questions, such that whether a particular algorithm can do the perfect packing. (However, the author does not claim that the actual computations would be any easier.)

We conclude with a result on the lower bound of $\eta_0$, which was somewhat unexpected at first glance. Let $\Delta{A}=(\Delta{x},0,...,0)$, $A'=A+\Delta{A}$. Let $\Delta{x}>0$ and $\eta_0(A')<\eta_0(A)$. We construct as in Lemma \ref{t_0}. Firstly, $|\Delta\eta_0|=-\Delta\eta_0\leqslant\eta_0(A)-\eta(A'_w)\leqslant{S}/T_0-(S+l_1\Delta{x})/(T_0+q_0\Delta{x})=(Sq_0-T_0l_1)\Delta{x}/(T_0(T_0+q_0\Delta{x}))\leqslant(Sq_0-T_0l_1)\Delta{x}/(T_0q_0\Delta{x})=S/T_0-l_1/q_0=\eta_0-l_1/q_0$, i.e.\ $\eta_0(A)-\eta_0(A')\leqslant\eta_0(A)-l_1/q(A_0)$. We get $\eta_0(A')\geqslant{l_1}/q(A_0)$. Secondly, since $0\leqslant|\Delta{\eta_0}|$, we have $\eta_0(A)\geqslant{l_1}/q(A_0)$. Now for all $\Delta{x}>0$, if $\eta_0(A+\Delta{A})$ is larger than that of $A$, it is larger than $l_1/q(A_0)$, which follows from the second argument. If it is smaller, it is still larger than $l_1/q(A_0)$, which is the first argument. To summarize, if there exists an $A'$ such that $\Delta{x}>0$ and $\eta_0(A')<\eta_0(A)$, we have a \emph{constant} lower bound for $\eta_0(A+\Delta{A})$ for all $\Delta{x}>0$. (If there does not exist such $A'$, that would simply mean $\eta_0(A+\Delta{A})$ is bounded from below by $\eta_0(A)$ for $\Delta{x}>0$.) (As a corollary, whenever rectangles expand, the packing efficiency can never approach zero.)

\section{Acknowledgement}

I am grateful to Professor Ma who raised valuable comments and questions on the manuscript.

\end{document}